\documentclass{article}
\usepackage{amsmath}
\usepackage{amsfonts}
\usepackage{latexsym}
\usepackage{amsthm}
\usepackage{mathrsfs}
\usepackage{amscd}
\usepackage{amsmath}
\usepackage{amssymb}
\usepackage[all, warning]{onlyamsmath}

\makeatletter
\@addtoreset{equation}{section}

\makeatother

\newcommand{\Hom}{\mathop{\mathrm{Hom}}\nolimits}
\newcommand{\Ker}{\mathop{\mathrm{Ker}}\nolimits}

\newcommand{\Spec}{\mathop{\mathrm{Spec}}\nolimits}

\newcommand{\Mod}{\mathop{\mathcal{M}od_k}\nolimits}
\newcommand{\ModA}{\mathop{\mathcal{M}od_k^{\mathbb{A}^1}}\nolimits}
\newcommand{\Ab}{\mathop{\mathcal{A}b_k}\nolimits}
\newcommand{\AbA}{\mathop{\mathcal{A}b_k^{\mathbb{A}^1}}\nolimits}
\newcommand{\Smk}{\mathop{\mathcal{S}m_k}\nolimits}
\newcommand{\Spck}{\mathop{\mathcal{S}pc_k}\nolimits}

\newtheorem{thm}{Theorem}[section]
\newtheorem{lem}[thm]{Lemma}
\newtheorem{prop}[thm]{Proposition}

\theoremstyle{definition}
\newtheorem{defi}[thm]{Definition}
\newtheorem{examp}[thm]{Example}
\newtheorem{rem}[thm]{Remark}

\newtheorem*{nota}{Notation}
\newtheorem*{acknow}{Acknowledgments}

\begin{document}

\title{Notes on $\mathbb{A}^1$-contractibility and $\mathbb{A}^1$-excision}
\author{Yuri Shimizu\footnote{Department of Mathematics, Tokyo Institute of Technology \newline E-mail: qingshuiyouli27@gmail.com}}
\maketitle

\begin{abstract}
We prove that a smooth scheme of dimension $n$ over a perfect field is $\mathbb{A}^1$-weakly equivalent to a point if it is $\mathbb{A}^1$-$n$-connected. We also prove an excision result for $\mathbb{A}^1$-homotopy sheaves over a perfect field. 
\end{abstract}



\section{Introduction}

In this paper, we discuss $n$-connectedness of schemes over a field $k$ in the sense of $\mathbb{A}^1$-homotopy theory. Morel-Voevodsky \cite{MV} defines an $\mathbb{A}^1$-homotopy version of homotopy groups or sets as sheaves over the large Nisnevich site on the category of smooth $k$-schemes $\Smk$. They are called the $\mathbb{A}^1$-homotopy sheaves and denoted by $\pi_i^{\mathbb{A}^1}(X)$ for $X \in \Smk$. A scheme $X$ is called $\mathbb{A}^1$-$n$-connected if $\pi_i^{\mathbb{A}^1}(X)$ is isomorphic to the constant sheaf valued on a point for all $0 \leq i \leq n$. $\mathbb{A}^1$-$0$-connected schemes are simply called $\mathbb{A}^1$-connected. 

In $\mathbb{A}^1$-homotopy theory, schemes are considered up to $\mathbb{A}^1$-weak equivalences. We say that a scheme over $k$ is $\mathbb{A}^1$-contractible if it is $\mathbb{A}^1$-weak equivalence to $\Spec k$. Morel-Voevodsky \cite{MV} proved the Whitehead theorem for $\mathbb{A}^1$-homotopy, which says that a scheme is $\mathbb{A}^1$-contractible if it is $\mathbb{A}^1$-$n$-connected for \textit{all} non-negative integer $n$. Our first result is an improvement of this theorem, showing that the range of $n$ can be actually restricted to $0 \leq n \leq \dim X$. More precisely, we prove the following.

\begin{thm}[see Theorem \ref{main1}]\label{int-main1}
Let $k$ be a perfect field and $X$ a pointed smooth $k$-scheme of dimension $n \geq 1$. If $X$ is $\mathbb{A}^1$-$n$-connected, then it is $\mathbb{A}^1$-contractible. 
\end{thm}

The proof of Theorem \ref{int-main1} is based on Morel's theory of $\mathbb{A}^1$-homology sheaves \cite{Mo2}, which is an $\mathbb{A}^1$-version of singular homology. We give a relation between the $\mathbb{A}^1$-homology and the Nisnevich cohomology, generalizing Asok's argument \cite{As} from degree $0$ to all degrees. By combining this with the vanishing result for Nisnevich cohomology and using Morel's $\mathbb{A}^1$-Hurewicz theorem \cite{Mo2}, we obtain Theorem \ref{int-main1}.




Our second result concerns excision of $\mathbb{A}^1$-homotopy sheaves. Let $X$ be a pointed smooth scheme over a field $k$ and $U$ a pointed open subscheme of $X$ whose complement has codimension $d \geq 2$. Asok-Doran \cite{AD} proved, for $k$ infinite field, that if $X$ is $\mathbb{A}^1$-$(d-3)$-connected, then the canonical morphism
\begin{equation*}
\pi_{l}^{\mathbb{A}^1}(U) \to \pi_{l}^{\mathbb{A}^1}(X)
\end{equation*}
is an isomorphism for $0 \leq l \leq d-2$ and an epimorphism for $l=d-1$. Our second result is a weaker version of this theorem over a perfect field. This covers the finite field case which is not treated in Asok-Doran's result. We need, however, an extra assumption that first $\mathbb{A}^1$-homotopy sheaves are abelian. To be precise, we prove the following.


\begin{thm}[see Theorem \ref{main2}]\label{int-excision}
Let $k$ be a perfect field, $U$ a pointed open subscheme of a pointed smooth $k$-scheme $X$ whose complement has codimension $d \geq 2$, and $l$ a non-negative integer. Assume that $U$ and $X$ are $\mathbb{A}^1$-connected and that $\pi_{1}^{\mathbb{A}^1}(U)$ and $\pi_{1}^{\mathbb{A}^1}(X)$ are abelian. If $X$ is $\mathbb{A}^1$-$(l-1)$-connected, then the canonical morphism
\begin{equation*}
\pi_{l}^{\mathbb{A}^1}(U) \to \pi_{l}^{\mathbb{A}^1}(X)
\end{equation*}
is an isomorphism when $0 \leq l \leq d-2$ and an epimorphism when $l=d-1$.
\end{thm}

For the proof of this theorem, we first establish an $\mathbb{A}^1$-excision theorem for $\mathbb{A}^1$-homology sheaves over an arbitrary field, which is proved for degree $0$ by Asok \cite{As}. Theorem \ref{int-excision} is a homotopy version of this via the $\mathbb{A}^1$-Hurewicz theorem \cite{Mo2}.


The paper is organized as follows. In Section \ref{recoll}, we recall basic facts on $\mathbb{A}^1$-homotopy theory. In Section \ref{dim}, we prove Theorem \ref{int-main1}. In Section \ref{codim2}, we prove \ref{int-excision}. In Section \ref{adj-inc}, we give a refinement of Morel's $\mathbb{A}^1$-Hurewicz theorem \cite{Mo2}. This proves a weaker version of Theorem \ref{int-main1} for not perfect field $k$ (see Theorem \ref{Finalcor}).

\begin{nota}
Throughout this paper, we fix a field $k$ and a commutative unital ring $R$. If we say $k$-{\it scheme}, it means a separated $k$-scheme of finite type. We denote $\Smk$ for the category of smooth $k$-schemes. In this paper, every sheaf is considered on the large Nisnevich site over $\Smk$.
\end{nota}

\begin{acknow}
I would like to thank my adviser Shohei Ma for many useful advices. I would also like to thank Aravind Asok for helpful comments.
\end{acknow}


\section{Recollections of unstable $\mathbb{A}^1$-homotopy theory}\label{recoll}

In this section, we recall some basic facts on unstable $\mathbb{A}^1$-homotopy theory following \cite{MV}, \cite{Mo2}, \cite{CD12}, and \cite{As}. $\mathbb{A}^1$-homotopy theory is constructed by using simplicial sets. We refer to \cite{GJ} for the homotopy theory of simplicial sets. Let $\Spck$ be the category of simplicial Nisnevich sheaves of sets. An object of $\Spck$ is called a $k${\it-space}. Via the Yoneda embedding $\Smk \to \Spck$, smooth $k$-schemes are naturally regarded as $k$-spaces.

\subsection{$\mathbb{A}^1$-model structure}\label{BMS}

Morel-Voevodsky \cite{MV} constructed two model structures for $\Spck$. The first one is the simplicial model structure, which is the Nisnevich sheafification of the homotopy theory of simplicial sets. A morphism of $k$-spaces $f:\mathcal{X} \to \mathcal{Y}$ is called a \textit{simplicial weak equivalence} if for an arbitrary point $x$ of every $X \in \Smk$ the induced morphism $\mathcal{X}_x \to \mathcal{Y}_x$ of stalks is a weak equivalence of simplicial sets.

\begin{thm}[Morel-Voevodsky {\cite[Thm. 1.4.]{MV}}]\label{smodel}
The category $\Spck$ has a model structure in which weak equivalences are simplicial weak equivalences and cofibrations are monomorphisms.
\end{thm}

This model structure is called the \textit{simplicial model structure}. We denote $\mathcal{H}_s(k)$ for its homotopy category and write $[\mathcal{X},\mathcal{Y}]_s = \Hom_{\mathcal{H}_s(k)}(\mathcal{X},\mathcal{Y})$ for $k$-spaces $\mathcal{X}$ and $\mathcal{Y}$.

Next, $\mathbb{A}^1$-model structure for $\Spck$ is constructed via the simplicial model structure as follows. A $k$-space $\mathcal{X}$ is called $\mathbb{A}^1$-\textit{local} if for every $k$-spaces $\mathcal{Y}$ the natural map $[\mathcal{Y},\mathcal{X}]_s \to [\mathcal{Y} \times \mathbb{A}^1,\mathcal{X}]_s$ is bijective. Moreover, a morphism of $k$-spaces $f:\mathcal{X} \to \mathcal{Y}$ is called an $\mathbb{A}^1$-\textit{weak equivalence} if for every $\mathbb{A}^1$-local $k$-space $\mathcal{Z}$ the induced map $[\mathcal{Y},\mathcal{Z}]_s \to [\mathcal{X},\mathcal{Z}]_s$ is bijective.

\begin{thm}[Morel-Voevodsky {\cite[Thm. 3.2.]{MV}}]\label{A1model}
The category $\Spck$ has a model structure in which weak equivalences are $\mathbb{A}^1$-weak equivalences and cofibrations are monomorphisms.
\end{thm}

This model structure is called the $\mathbb{A}^1$\textit{-model structure}. Unstable $\mathbb{A}^1$-homotopy theory is the homotopy theory for this model structure. We denote $\mathcal{H}_{\mathbb{A}^1}(k)$ for its homotopy category and write $[\mathcal{X},\mathcal{Y}]_{\mathbb{A}^1} = \Hom_{\mathcal{H}_{\mathbb{A}^1}(k)}(\mathcal{X},\mathcal{Y})$ for $k$-spaces $\mathcal{X}$ and $\mathcal{Y}$. 

\subsection{$\mathbb{A}^1$-homotopy sheaves}

We recall basic facts of $\mathbb{A}^1$-homotopy sheaves introduced by Morel-Voevodsky \cite{MV}. For a $k$-space $\mathcal{X}$, the $\mathbb{A}^1$\textit{-connected components of $\mathcal{X}$}, denoted by $\pi_0^{\mathbb{A}^1}(\mathcal{X})$, is defined as the Nisnevich sheafification of the presheaf
\begin{equation*}
U \mapsto [U,\mathcal{X}]_{\mathbb{A}^1}.
\end{equation*}
The $k$-space $\mathcal{X}$ is called \textit{$\mathbb{A}^1$-connected} if $\pi_0^{\mathbb{A}^1}(\mathcal{X}) \cong \Spec k$. By the following theorem, the $\mathbb{A}^1$-weak equivalence class of a $k$-space can be represented by an $\mathbb{A}^1$-local space.

\begin{thm}[Morel-Voevodsky {\cite[Lem. 3.20 and Lem. 3.21.]{MV}}]\label{A1-resol}
There exists an endofunctor $\mathrm{Ex}_{\mathbb{A}^1}$ of $\Spck$ and a natural transformation $\theta:\mathrm{id} \to \mathrm{Ex}_{\mathbb{A}^1}$ such that $\mathrm{Ex}_{\mathbb{A}^1}(\mathcal{X})$ is $\mathbb{A}^1$-local and $\mathcal{X} \to \mathrm{Ex}_{\mathbb{A}^1}(\mathcal{X})$ is an $\mathbb{A}^1$-weak equivalence for every $\mathcal{X} \in \Spck$.
\end{thm}

A pair of a $k$-space $\mathcal{X}$ and a morphism $x:\Spec k \to \mathcal{X}$ is called a \textit{pointed $k$-space}. We often denote it simply by $\mathcal{X}$. For a pointed $k$-space $\mathcal{X}$, we also regard $\mathrm{Ex}_{\mathbb{A}^1}(\mathcal{X})$ as a pointed space by the point $\Spec k \to \mathcal{X} \to \mathrm{Ex}_{\mathbb{A}^1}(\mathcal{X})$. For a pointed simplicial set $S$ and positive integer $i$, we denote $\pi_i(S)$ for the $i$-th simplicial homotopy group.

For a pointed $k$-space $\mathcal{X}$ and a positive integer $i$, \textit{$\mathbb{A}^1$-homotopy sheaf} of $\mathcal{X}$, say $\pi_i^{\mathbb{A}^1}(\mathcal{X})$, is defined as the Nisnevich sheafification of the presheaf
\begin{equation*}
U \mapsto \pi_i(\mathrm{Ex}_{\mathbb{A}^1}(\mathcal{X})(U)).
\end{equation*}
A $k$-space $\mathcal{X}$ is called \textit{$\mathbb{A}^1$-$n$-connected} if it is $\mathbb{A}^1$-connected and $\pi_i^{\mathbb{A}^1}(\mathcal{X}) = 0$ for all $0 < i \leq n$. Especially, $\pi_1^{\mathbb{A}^1}(\mathcal{X})$ is called the \textit{$\mathbb{A}^1$-fundamental sheaf}, and ``$\mathbb{A}^1$-$1$-connected'' is also rephrased as ``\textit{$\mathbb{A}^1$-simply connected}''. If we say ``$\mathbb{A}^1$-$(-1)$-connected'' it means ``non-empty''. 

$\mathbb{A}^1$-homotopy sheaves are $\mathbb{A}^1$-weakly homotopy invariant. Indeed, a morphism of $k$-spaces $\mathcal{X} \to \mathcal{Y}$ is an $\mathbb{A}^1$-weak equivalence if and only if $\mathrm{Ex}_{\mathbb{A}^1}(\mathcal{X}) \to \mathrm{Ex}_{\mathbb{A}^1}(\mathcal{Y})$ is a simplicial weak equivalence by Yoneda's lemma in $\mathcal{H}_s(k)$. Therefore, $\mathbb{A}^1$-weak equivalences induce isomorphisms for all stalks of $\mathbb{A}^1$-homotopy sheaves.


A Nisnevich sheaf of groups $G$ is called \textit{strongly $\mathbb{A}^1$-invariant} (\cite{Mo2}) if for every $X \in \Smk$ and $i \in \{0,1\}$ the projection $X \times \mathbb{A}^1 \to X$ induces a bijection
\begin{equation*}
H_{Nis}^i(X,G) \to H_{Nis}^i(X\times \mathbb{A}^1,G).
\end{equation*}
We denote $\mathcal{G}r^{\mathbb{A}^1}_k$ for the category of strongly $\mathbb{A}^1$-invariant sheaves. By \cite[Thm. 5.1]{Mo2}, $\mathbb{A}^1$-fundamental sheaves are strongly $\mathbb{A}^1$-invariant.

\subsection{$\mathbb{A}^1$-homology theory}

Morel \cite{Mo2} introduced an $\mathbb{A}^1$-version of singular homology theory, called $\mathbb{A}^1$-homology theory, with coefficients $\mathbb{Z}$. By using the theory of Cisinski and \'Deglise \cite{CD12}, Asok \cite{As} constructs its generalization for coefficients every commutative unital ring $R$.

For an abelian category $\mathscr{A}$, let $\mathcal{C}_{\bullet}(\mathscr{A})$ be the unbounded chain complexes of degree $-1$ in $\mathscr{A}$, and $\mathcal{D}_{\bullet}(\mathscr{A})$ the unbounded derived category of $\mathscr{A}$, \textit{i.e.}, the localization of $\mathcal{C}_{\bullet}(\mathscr{A})$ by quasi-isomorphisms. Every object of $\mathscr{A}$ is viewed as a complex concentrated in degree zero. For a chain complex $A_{\bullet} \in \mathcal{C}_{\bullet}(\mathscr{A})$, we write $A^{\bullet}$ for the cochain complex $A^i = A_{-i}$. Let $\Mod(R)$ be the category of Nisnevich sheaves of $R$-modules. We especially write $\Ab = \Mod(\mathbb{Z})$. 

Following \cite{CD12}, a chain complex $A_{\bullet} \in \mathcal{C}_{\bullet}(\Mod(R))$ is called $\mathbb{A}^1$\textit{-local} if for every smooth $k$-scheme $X$ and every $i \in \mathbb{Z}$ the projection $X \times \mathbb{A}^1 \to X$ induces an isomorphism between Nisnevich hypercohomology groups
\begin{equation*}
\mathbb{H}^i_{Nis}(X,A^{\bullet}) \to \mathbb{H}^i_{Nis}(X \times \mathbb{A}^1,A^{\bullet}).
\end{equation*}
If $M \in \Mod(R)$ is $\mathbb{A}^1$-local as a complex, it is called \textit{strictly $\mathbb{A}^1$-invariant}. We denote $\ModA(R)$ for the full subcategory of $\Mod(R)$ consisting of strictly $\mathbb{A}^1$-invariant sheaves. We especially write $\AbA = \ModA(\mathbb{Z})$. A morphism $f: B_{\bullet} \to C_{\bullet}$ in $\mathcal{C}_{\bullet}(\Mod(R))$ is called an $\mathbb{A}^1$\textit{-quasi-isomorphism} if for every $\mathbb{A}^1$-local complex $A_{\bullet}$ the induced map
\begin{equation*}
f^* : \Hom_{\mathcal{D}_{\bullet}(\Mod(R))}(C_{\bullet},A_{\bullet}) \to \Hom_{\mathcal{D}_{\bullet}(\Mod(R))}(B_{\bullet},A_{\bullet})
\end{equation*}
is an isomorphism.


\begin{thm}[Cisinski-D\'{e}glise {\cite[Cor. 1.1.17]{CD12}}]\label{A1-localization}
Let $\mathcal{D}_{\bullet}(\Mod(R))^{\mathbb{A}^1-loc}$ be the full subcategory of $\mathcal{D}_{\bullet}(\Mod(R))$ consisting of $\mathbb{A}^1$-local complexes. Then the inclusion
\begin{equation*}
\mathcal{D}_{\bullet}(\Mod(R))^{\mathbb{A}^1-loc} \hookrightarrow \mathcal{D}_{\bullet}(\Mod(R))
\end{equation*}
admits a left adjoint functor
\begin{equation*}
L_{\mathbb{A}^1}:\mathcal{D}_{\bullet}(\Mod(R)) \to \mathcal{D}_{\mathbb{A}^1}(k,R) \simeq \mathcal{D}_{\bullet}(\Mod(R))^{\mathbb{A}^1-loc},
\end{equation*}
which is called the \textit{$\mathbb{A}^1$-localization functor}.
\end{thm}

Theorem \ref{A1-localization} is for the derived category of cochain complexes $\mathcal{D}^{\bullet}(\Mod(R))$. By applying the isomorphism $\mathcal{D}_{\bullet}(\Mod(R)) \to \mathcal{D}^{\bullet}(\Mod(R)); A_{\bullet} \mapsto A^{\bullet}$, we also obtain a homological version.

For $\mathcal{X} \in \Spck$, we define the simplicial Nisnevich sheaf of $R$-modules $R(\mathcal{X})$ by that the $n$-skeleton $R(X)_n$ is the sheafification of the presheaf 
\begin{equation*}
U \mapsto \bigoplus_{s \in \mathcal{X}_n(U)}R.
\end{equation*}
We write $C(\mathcal{X},R)$ for the normalized chain complex of $R(X)$. For $\mathcal{X} \in \Spck$, the $\mathbb{A}^1$-local complex 
\begin{equation*}
C^{\mathbb{A}^1}(\mathcal{X},R) = L_{\mathbb{A}^1}(C(\mathcal{X},R))
\end{equation*}
is called the \textit{$\mathbb{A}^1$-chain complex of $\mathcal{X}$ with coefficients $R$}. The \textit{$\mathbb{A}^1$-homology sheaf} $\mathbf{H}_{i}^{\mathbb{A}^1}(\mathcal{X},R)$ is defined as the homology of $C^{\mathbb{A}^1}(\mathcal{X},R)$, which is a Nisnevich sheaf. By \cite[Thm. 5.22 and Cor. 5.23]{Mo2}, $\mathbf{H}_{i}^{\mathbb{A}^1}(\mathcal{X},R)$ is trivial for all $i<0$ and strictly $\mathbb{A}^1$-invariant for all $i \geq 0$. We especially write $C^{\mathbb{A}^1}(\mathcal{X})=C^{\mathbb{A}^1}(\mathcal{X},\mathbb{Z})$ and $\mathbf{H}_{i}^{\mathbb{A}^1}(\mathcal{X})=\mathbf{H}_{i}^{\mathbb{A}^1}(\mathcal{X},\mathbb{Z})$. Finally, the \textit{reduced $\mathbb{A}^1$-homology sheaf} $\tilde{\mathbf{H}}_{i}^{\mathbb{A}^1}(\mathcal{X},R)$ is defined as the kernel of the natural morphism $\mathbf{H}_{i}^{\mathbb{A}^1}(\mathcal{X},R) \to \mathbf{H}_{i}^{\mathbb{A}^1}(\Spec k,R)$.

\begin{examp}[Asok {\cite[Example 2.6]{As}}]\label{A^1homspeck}
The zeroth $\mathbb{A}^1$-homology sheaf $\mathbf{H}_{0}^{\mathbb{A}^1}(\Spec k,R)$ is isomorphic to the constant sheaf of $R$ and the $i$-th $\mathbb{A}^1$-homology sheaf $\mathbf{H}_{i}^{\mathbb{A}^1}(\Spec k,R)$ is trivial for all $i \geq 1$. Therefore, $C^{\mathbb{A}^1}(\Spec k,R)$ is quasi-isomorphic to $R(\Spec k)$.
\end{examp}


\subsection{Eirenberg-Maclane spaces and adjunctions}\label{adj}

We denote $\Delta^{op} \Mod(R)$ for the category of simplicial objects of $\Mod(R)$ and write $\mathcal{C}_{\geq0}(\Mod(R))$ for the category of chain complexes of $\Mod(R)$ supported in degree $\geq0$. The Dold-Kan correspondence gives an equivalence of categories $K:\mathcal{C}_{\geq0}(\Mod(R)) \to \Delta^{op} \Mod(R)$. Morel-Voevodsky \cite{MV} defined the \textit{Eirenberg-Maclane space} $K(M,n)$ of $M \in \Mod(R)$ as $K(M[n])$.

\begin{lem}[Asok {\cite[(2.1)]{As}}]\label{E-Madj}
For $\mathcal{X} \in \Smk$, $M \in \Mod(R)$ and a non-negative integer $n$, there exists a natural isomorphism
\begin{equation*}
[\mathcal{X},K(M.n)]_s \cong \Hom_{\mathcal{D}_{\bullet}(\Mod(R))}(C(\mathcal{X},R),M[n]).
\end{equation*}
\end{lem}

For an abelian category $\mathscr{A}$, we write $\mathcal{D}_{\geq0}(\mathscr{A})$ (resp. $\mathcal{D}_{\leq0}(\mathscr{A})$) for the full subcategory of $\mathcal{D}_{\bullet}(\mathscr{A})$ consisting complexes $C_{\bullet}$ such that $C_i = 0$ for all $i<0$ (resp. $i > 0$). We will use the following well-known adjunction:

\begin{lem}\label{AsH0}
For an abelian category $\mathscr{A}$, the zeroth homology functor $H_0:\mathcal{D}_{\geq0}(\mathscr{A}) \to \mathscr{A}$ is left adjoint to the natural embedding $\mathscr{A} \to \mathcal{D}_{\geq0}(\mathscr{A})$.
\end{lem}

\begin{proof}
Since the truncation functor $\tau_{\leq 0} : \mathcal{D}_{\bullet}(\mathscr{A}) \to \mathcal{D}_{\leq0}(\mathscr{A})$ is left adjoint to the inclusion $\mathcal{D}_{\leq0}(\mathscr{A}) \hookrightarrow \mathcal{D}_{\bullet}(\mathscr{A})$, we have a natural isomorphism
\begin{equation*}
\Hom_{\mathcal{D}_{\geq0}(\mathscr{A})}(C_{\bullet},A) \cong \Hom_{\mathcal{D}_{\leq0}(\mathscr{A})}(\tau_{\leq 0}(C_{\bullet}),A) = Hom_{\mathscr{A}}(H_0(C_{\bullet}),A)
\end{equation*}
for $A \in \mathscr{A}$ and $C_{\bullet} \in \mathcal{D}_{\geq0}(\mathscr{A})$.
\end{proof}


\section{$\mathbb{A}^1$-contractibility and dimension}\label{dim}

In this section, we prove Theorem \ref{int-main1}. A $k$-space $\mathcal{X}$ is called \textit{$\mathbb{A}^1$-contractible}, if the canonical morphism $\mathcal{X} \to \Spec k$ is an $\mathbb{A}^1$-weak equivalence. Morel-Voevodsky proved the Whitehead theorem for $\mathbb{A}^1$-homotopy (see \cite[Prop. 2.1.4]{MV}), which says that a scheme $X$ is $\mathbb{A}^1$-contractible if and only if it is $\mathbb{A}^1$-$n$-connected for all $n \geq 0$. Our theorem is an improvement of this theorem, showing that the range of $n$ can be actually restricted to $0 \leq n \leq \dim X$. We first prepare a lemma which relates the $\mathbb{A}^1$-homology sheaf to the Nisnevich cohomology group. This is a generalization of a result of Asok \cite[Lem. 3.3]{As} where the case $n=0$ is proved.

\begin{lem}\label{key}
Let $X$ be a smooth $k$-scheme with a $k$-rational point, $M$ a strictly $\mathbb{A}^1$-invariant sheaf of $R$-modules, and $n \geq 0$.  If $\tilde{\mathbf{H}}_i^{\mathbb{A}^1}(X,R)=0$ for every $i<n$, then there exists a natural isomorphism
\begin{equation*}
H^n_{Nis}(X,M) \to \Hom_{\ModA(R)}(\mathbf{H}_n^{\mathbb{A}^1}(X,R),M).
\end{equation*}
\end{lem}

\begin{proof}
Since the case $n=0$ is proved by \cite[Lem. 3.3]{As}, we assume $n \geq 1$. By \cite[Prop. 1.26]{MV}, there exists a natural isomorphism 
\begin{equation*}
H^n_{Nis}(X,M) \cong [X,K(M,n)]_s.
\end{equation*}
The right-hand side is isomorphic to $\Hom_{\mathcal{D}_{\bullet}(\Mod(R))}(C(X,R),M[n])$ by Lemma \ref{E-Madj}. Since $M[n]$ is $\mathbb{A}^1$-local, we also have
\begin{equation*}
\Hom_{\mathcal{D}_{\bullet}(\Mod(R))}(C(X,R),M[n]) \cong \Hom_{\mathcal{D}_{\bullet}(\Mod(R))}(C^{\mathbb{A}^1}(X,R),M[n])
\end{equation*}
by the definition of $C^{\mathbb{A}^1}(X,R)$. Therefore, we obtain a natural isomorphism
\begin{equation}\label{eq-Nis-HomD}
H^n_{Nis}(X,M) \cong \Hom_{\mathcal{D}_{\bullet}(\Mod(R))}(C^{\mathbb{A}^1}(X,R),M[n]).
\end{equation}
By Example \ref{A^1homspeck} and $n>0$, this isomorphism gives 
\begin{align*}\label{eq-Nis-HomDspeck}
\Hom_{\mathcal{D}_{\bullet}(\Mod(R))}(R,M[n]) &\cong \Hom_{\mathcal{D}_{\bullet}(\Mod(R))}(C^{\mathbb{A}^1}(\Spec k,R),M[n])\\
&\cong H^n_{Nis}(\Spec k,M)\\
&= 0.
\end{align*}
Since $X$ has a $k$-rational point, we have $C^{\mathbb{A}^1}(X,R)=\tilde{C}^{\mathbb{A}^1}(X,R)\oplus R$. Hence,
\begin{align*}
&\Hom_{\mathcal{D}_{\bullet}(\Mod(R))}(C^{\mathbb{A}^1}(X,R),M[n])\\
&=\Hom_{\mathcal{D}_{\bullet}(\Mod(R))}(\tilde{C}^{\mathbb{A}^1}(X,R),M[n]) \oplus \Hom_{\mathcal{D}_{\bullet}(\Mod(R))}(R,M[n])\\
&=\Hom_{\mathcal{D}_{\bullet}(\Mod(R))}(\tilde{C}^{\mathbb{A}^1}(X,R),M[n]).
\end{align*}
Thus, by the isomorphism \eqref{eq-Nis-HomD}, we only need to show that if a complex $C \in \Mod(R)$ satisfies $H_i(C)=0$ for all $i<n$, then 
\begin{equation}\label{purposeisomorphism}
\Hom_{\mathcal{D}_{\bullet}(\Mod(R))}(C,M[n]) \cong \Hom_{\Mod(R)}(H_n(C),M).
\end{equation}
By our assumption, the complex $C$ is quasi-isomorphic to the complex
\begin{equation*}
\tau_{\geq n}(C) : \cdots \to C_{n+2} \to C_{n+1} \to \Ker(C_{n} \to C_{n-1}) \to 0 \to \cdots.
\end{equation*}
Then we have 
\begin{align*}
\Hom_{\mathcal{D}_{\bullet}(\Mod(R))}(C,M[n]) &\cong \Hom_{\mathcal{D}_{\bullet}(\Mod(R))}(\tau_{\geq n}(C),M[n])\\
&\cong \Hom_{\mathcal{D}_{\bullet}(\Mod(R))}(\tau_{\geq n}(C)[-n],M).
\end{align*}
Since $\tau_{\geq n}(C)[-n] \in \mathcal{D}_{\geq 0}(\Mod(R))$, Lemma \ref{AsH0} gives an isomorphism
\begin{equation*}
\Hom_{\mathcal{D}_{\bullet}(\Mod(R))}(\tau_{\geq n}(C)[-n],M) \cong \Hom_{\Mod(R)}(H_0(\tau_{\geq n}(C)[-n]),M).
\end{equation*}
Since $H_0(\tau_{\geq n}(C)[-n]) \cong H_n(\tau_{\geq n}(C)) \cong H_n(C)$, we obtain the isomorphism \eqref{purposeisomorphism}.
\end{proof}

Now we prove Theorem \ref{int-main1}.

\begin{thm}\label{main1}
Assume $k$ perfect. Let $X$ be a pointed smooth $k$-scheme of dimension $n \geq 1$. If $X$ is $\mathbb{A}^1$-$n$-connected, then it is $\mathbb{A}^1$-contractible. 
\end{thm}

\begin{proof}
By the $\mathbb{A}^1$-Whitehead theorem \cite[Prop. 2.1.4]{MV}, the $k$-scheme $X$ is $\mathbb{A}^1$-weakly equivalent to $\Spec k$ if and only if $\pi_{i}^{\mathbb{A}^1}(X)$ is trivial for every $i \geq 0$. Therefore, it suffices to show that if $X$ is $\mathbb{A}^1$-$m$-connected for $m \geq n$, then it is $\mathbb{A}^1$-$(m+1)$-connected. When $X$ is $\mathbb{A}^1$-$m$-connected, Morel's $\mathbb{A}^1$-Hurewicz theorem \cite[Thm. 5.37]{Mo2} shows that $\tilde{\mathbf{H}}_i^{\mathbb{A}^1}(X)=0$ for $i \leq m$ and gives an isomorphism $\pi_{m+1}^{\mathbb{A}^1}(X) \cong \mathbf{H}_{m+1}^{\mathbb{A}^1}(X)$. Applying Lemma \ref{key} with $R=\mathbb{Z}$, we see that
\begin{equation*}
\Hom_{\AbA}(\pi_{m+1}^{\mathbb{A}^1}(X),M) \cong H^{m+1}_{Nis}(X,M)
\end{equation*}
for every $M \in \AbA$. On the other hand, $H^i_{Nis}(X,M) = 0$ when $i > \dim X = n$ (see \cite[Thm. 1.32]{Nis}). Therefore, $\Hom_{\AbA}(\pi_{m+1}^{\mathbb{A}^1}(X),M) = 0$ for every $M \in \AbA$. Since $\pi_{m+1}^{\mathbb{A}^1}(X) \in \AbA$ by \cite[Cor. 5.2]{Mo2}, Yoneda's lemma in $\AbA$ gives $\pi_{m+1}^{\mathbb{A}^1}(X) = 0$. 
\end{proof}

Lemma \ref{key} also has the following application.

\begin{rem}\label{rempic}
For $X \in \Smk$ with $\tilde{\mathbf{H}}_0^{\mathbb{A}^1}(X) = 0$, Lemma \ref{key} and \cite[Lem. 6.4.7]{Mo1} give an isomorphism
\begin{equation*}
\Hom_{\Ab}(\mathbf{H}_1^{\mathbb{A}^1}(X),\mathbb{G}_m) \cong \mathrm{Pic}(X).
\end{equation*}
By Morel's $\mathbb{A}^1$-Hurewicz theorem \cite[Thm. 5.35]{Mo2}, this isomorphism induces $\Hom_{\mathcal{G}r_k}(\pi^{\mathbb{A}^1}_1(X),\mathbb{G}_m) \cong \mathrm{Pic}(X)$. This is another proof of \cite[Prop. 5.1.4]{AM}.
\end{rem}


\section{$\mathbb{A}^1$-excision over a perfect field}\label{codim2}

In this section, we consider an $\mathbb{A}^1$-excision theorem over a perfect field. Let $U$ be $\mathbb{A}^1$-connected open subscheme of $\mathbb{A}^1$-connected $X \in \Smk$ whose complement has codimension $d$. The $\mathbb{A}^1$-excision theorem of Asok-Doran \cite[Thm. 4.1]{AD} states that if $k$ is infinite and $X$ is $\mathbb{A}^1$-$(d-3)$-connected, then the canonical morphism $\pi_{l}^{\mathbb{A}^1}(U) \to \pi_{l}^{\mathbb{A}^1}(X)$ is an isomorphism for $0 \leq l \leq d-2$ and an epimorphism for $l=d-1$. We prove a weaker version of this theorem over a perfect field:

\begin{thm}\label{main2}
Let $k$ be perfect, $U$ a pointed open subscheme of a pointed smooth $k$-scheme $X$ whose complement has codimension $d \geq 2$, and $l$ a non-negative integer. Assume that $U$ and $X$ are $\mathbb{A}^1$-connected and that $\pi_{1}^{\mathbb{A}^1}(U)$ and $\pi_{1}^{\mathbb{A}^1}(X)$ are abelian. If $X$ is $\mathbb{A}^1$-$(l-1)$-connected, then the canonical morphism
\begin{equation*}
\pi_{l}^{\mathbb{A}^1}(U) \to \pi_{l}^{\mathbb{A}^1}(X)
\end{equation*}
is an isomorphism when $0 \leq l \leq d-2$ and an epimorphism when $l=d-1$.
\end{thm}

This theorem covers the finite field case which is not treated in Asok-Doran's $\mathbb{A}^1$-excision theorem. Moreover, the assumption of the $\mathbb{A}^1$-connectedness of $X$ is weaker than the Asok-Doran's result which needs that $X$ is $\mathbb{A}^1$-$(d-3)$-connected. However, we need an extra assumption that $\pi_{1}^{\mathbb{A}^1}(U)$ and $\pi_{1}^{\mathbb{A}^1}(X)$ are abelian. We first prove an $\mathbb{A}^1$-excision theorem for $\mathbb{A}^1$-homology sheaves over an arbitrary field.

\begin{lem}\label{Hcodim}
Let $U$ be an open subscheme of a smooth $k$-scheme $X$ whose compliment has codimension $d \geq 2$ and $l$ a non-negative integer. Assume $U(k) \neq \emptyset$. If $\tilde{\mathbf{H}}_i^{\mathbb{A}^1}(X,R)=0$ for every $i < l$, then the canonical morphism 
\begin{equation*}
{\mathbf{H}}_{l}^{\mathbb{A}^1}(U,R) \to {\mathbf{H}}_{l}^{\mathbb{A}^1}(X,R)
\end{equation*}
is an isomorphism when $l \leq d-2$ and an epimorphism when $l = d-1$.
\end{lem}

\begin{proof}
It suffices to show for the case $l \leq d-1$. We use induction on $l$. The case of $l=0$ is proved by \cite[Prop. 3.8]{As}. We suppose that the lemma holds for $l \geq 1$ and $\tilde{\mathbf{H}}_i^{\mathbb{A}^1}(X,R)=0$ for every $i < l+1$. We may assume $l+1 \leq d-1$, or equivalently, $l \leq d-2$. Then we have 
\begin{equation*}
\tilde{\mathbf{H}}_i^{\mathbb{A}^1}(U,R) \cong \tilde{\mathbf{H}}_i^{\mathbb{A}^1}(X,R) = 0
\end{equation*}
for every $i \leq l$ by the inductive hypothesis. Therefore, Lemma \ref{key} gives isomorphisms
\begin{align*}
H^{l+1}_{Nis}(X,M) \cong \Hom_{\ModA(R)}(\mathbf{H}_{l+1}^{\mathbb{A}^1}(X,R),M),\\
H^{l+1}_{Nis}(U,M) \cong \Hom_{\ModA(R)}(\mathbf{H}_{l+1}^{\mathbb{A}^1}(U,R),M),
\end{align*}
for every $M \in \ModA(R)$. On the other hand, \cite[Lem. 6.4.4]{Mo1} proves that the natural morphism
\begin{equation*}
H^{l+1}_{Nis}(X,M) \to H^{l+1}_{Nis}(U,M)
\end{equation*}
is an isomorphism for $l+1 \leq d-2$ and a monomorphism for $l+1 = d-1$. Thus, the conclusion follows from Yoneda's lemma in $\ModA(R)$.
\end{proof}

Next, we prove Theorem \ref{main2}.

\begin{proof}[Proof of Theorem \ref{main2}]
It suffices to consider the case $l \leq d-1$. We use induction on $n$. Since $U$ and $X$ are $\mathbb{A}^1$-connected, canonical morphism $\pi_{0}^{\mathbb{A}^1}(U) \to \pi_{0}^{\mathbb{A}^1}(X)$ is an isomorphism. By \cite[Thm. 5.35 and 4.46]{Mo2}, we have isomorphisms
\begin{align*}
\pi_{1}^{\mathbb{A}^1}(U) \cong \mathbf{H}_{1}^{\mathbb{A}^1}(U),\\
\pi_{1}^{\mathbb{A}^1}(X) \cong \mathbf{H}_{1}^{\mathbb{A}^1}(X).
\end{align*}
because $\pi_{1}^{\mathbb{A}^1}(U)$ and $\pi_{1}^{\mathbb{A}^1}(X)$ are abelian. Therefore, Lemma \ref{Hcodim} shows that
\begin{equation*}
\pi_{1}^{\mathbb{A}^1}(U) \cong \mathbf{H}_{1}^{\mathbb{A}^1}(U) \to \mathbf{H}_{1}^{\mathbb{A}^1}(X) \cong \pi_{1}^{\mathbb{A}^1}(X)
\end{equation*}
is an isomorphism when $1 \leq d-2$ and an epimorphism when $1 = d-1$. This proves the theorem for $n=0$. Next, we suppose that the theorem holds for $l \geq 1$ and $X$ is $\mathbb{A}^1$-$l$-connected. We may assume $l+1 \leq d-1$, or equivalently, $l \leq d-2$. Then $U$ is also $\mathbb{A}^1$-$l$-connected, because the inductive hypothesis gives an isomorphism
\begin{equation*}
\pi_{i}^{\mathbb{A}^1}(U) \cong \pi_{i}^{\mathbb{A}^1}(X) = 0
\end{equation*}
for every $i \leq l$. Therefore, we have $\tilde{\mathbf{H}}_i^{\mathbb{A}^1}(U)=\tilde{\mathbf{H}}_i^{\mathbb{A}^1}(X)=0$ for every $i < l+1$ and
\begin{align*}
\pi_{l+1}^{\mathbb{A}^1}(U) \cong \mathbf{H}_{l+1}^{\mathbb{A}^1}(U)\\
\pi_{l+1}^{\mathbb{A}^1}(X) \cong \mathbf{H}_{l+1}^{\mathbb{A}^1}(X)
\end{align*}
by Morel's $\mathbb{A}^1$-Hurewicz theorem \cite[Thm. 5.37]{Mo2}. Thus, by Lemma \ref{Hcodim} again, we see that
\begin{equation*}
\pi_{l+1}^{\mathbb{A}^1}(U) \cong \mathbf{H}_{l+1}^{\mathbb{A}^1}(U) \to \mathbf{H}_{l+1}^{\mathbb{A}^1}(X) \cong \pi_{l+1}^{\mathbb{A}^1}(X)
\end{equation*}
is an isomorphism when $1 \leq d-2$ and an epimorphism when $1 = d-1$.
\end{proof}

Lemma \ref{Hcodim} also has the following application.

\begin{prop}
Let $U$ be an open subvariety of a smooth proper variety $X$ whose complement has codimension $\geq 2$. Then $U$ is not $\mathbb{A}^1$-simply connected.
\end{prop}

\begin{proof}
When $U$ is not $\mathbb{A}^1$-connected, this is clear. When $X$ is not $\mathbb{A}^1$-connected, $\tilde{\mathbf{H}}_{0}^{\mathbb{A}^1}(X) \neq 0$ by \cite[Thm. 4.14]{As}. Thus, \cite[Prop. 3.8]{As} shows $\tilde{\mathbf{H}}_{0}^{\mathbb{A}^1}(U) \neq 0$. Therefore, $U$ is not $\mathbb{A}^1$-connected. We may assume that $U$ and $X$ are $\mathbb{A}^1$-connected. Then $\tilde{\mathbf{H}}_{0}^{\mathbb{A}^1}(X) \cong \tilde{\mathbf{H}}_{0}^{\mathbb{A}^1}(U) = 0$ by \cite[Prop. 3.8]{As}. Therefore, Lemma \ref{Hcodim} induces an epimorphism
\begin{equation}\label{eqexamp}
{\mathbf{H}}_{1}^{\mathbb{A}^1}(U) \to {\mathbf{H}}_{1}^{\mathbb{A}^1}(X).
\end{equation}
On the other hand, we have an isomorphism
\begin{equation*}
\Hom_{\Ab}({\mathbf{H}}_{1}^{\mathbb{A}^1}(X),\mathbb{G}_m) \cong \mathrm{Pic}(X)
\end{equation*}
by Remark \ref{rempic}. Since $\mathrm{Pic}(X) \neq 0$ (see \cite[Prop. 5.1.4]{AM}), we have ${\mathbf{H}}_{1}^{\mathbb{A}^1}(X) \neq 0$. Since the epimorphism \eqref{eqexamp} shows ${\mathbf{H}}_{1}^{\mathbb{A}^1}(U) \neq 0$, the open set $U$ is not $\mathbb{A}^1$-simply connected.
\end{proof}


\section{$\mathbb{A}^1$-abelianization and the Hurewicz theorem}\label{adj-inc}

In this section, we consider  $\mathbb{A}^1$-Hurewicz theorem and Theorem \ref{main1} over a not perfect field. Since Morel's $\mathbb{A}^1$-Hurewicz theorem is only proved over perfect fields, so is Theorem \ref{main1}. We prove a generalization of the $\mathbb{A}^1$-Hurewicz theorem over an arbitrary field. Next, we prove a weaker version of Theorem \ref{main1} over an arbitrary field via the generalized $\mathbb{A}^1$-Hurewicz theorem. For this, we first construct a left adjoint functor of injections $\ModA(R) \hookrightarrow \Mod(R)$ and $\AbA \hookrightarrow \mathcal{G}r_k$.

\subsection{Left adjoint of $\ModA(R) \hookrightarrow \Mod(R)$ and $\mathbb{A}^1$-abelianization}

Let $M$ be a Nisnevich sheaf of $R$-modules. Then we write $M_{\mathbb{A}^1}$ for the strictly $\mathbb{A}^1$-invariant sheaf $H_0(L_{\mathbb{A}^1}(M))$.

\begin{lem}\label{adj:ModA-in-Mod}
The functor $(-)_{\mathbb{A}^1}:\Mod(R) \to \ModA(R)$ is left adjoint to the inclusion $\ModA(R) \hookrightarrow \Mod(R)$.
\end{lem}

\begin{proof}
For $M \in \Mod(R)$ and $N \in \Mod(R)$, Lemma \ref{AsH0} gives an isomorphism
\begin{align*}
\Hom_{\ModA(R)}(M_{\mathbb{A}^1},N) &= \Hom_{\Mod(R)}(H_0(L_{\mathbb{A}^1}(M)),N) \\
&\cong \Hom_{\mathcal{D}_{\bullet}(\Mod(R))}(L_{\mathbb{A}^1}(M),N) \\
&= \Hom_{\mathcal{D}_{\bullet}(\Mod(R))^{\mathbb{A}^1-loc}}(L_{\mathbb{A}^1}(M),N).
\end{align*}
By Theorem \ref{A1-localization}, we have
\begin{align*}
\Hom_{\mathcal{D}_{\bullet}(\Mod(R))^{\mathbb{A}^1-loc}}(L_{\mathbb{A}^1}(M),N) &\cong \Hom_{\mathcal{D}_{\bullet}(\Mod(R))}(M,N) \\
&= \Hom_{\Mod(R)}(M,N).
\end{align*}
\end{proof}

For integral coefficients, the adjunction
\begin{equation*}
\Hom_{\ModA(R)}(M_{\mathbb{A}^1},N) \cong \Hom_{\Mod(R)}(M,N)
\end{equation*}
of Lemma \ref{adj:ModA-in-Mod} comes from a natural morphism $M \to M_{\mathbb{A}^1}$ in $\Mod(R)$. Indeed, Morel's construction of the $\mathbb{A}^1$-localization functor comes from an endofunctor $L_{\mathbb{A}^1}^{ab}$ of $\mathcal{C}_{\bullet}(\Ab)$ and a natural transformation $\theta^{ab}:\mathrm{id} \to L_{\mathbb{A}^1}^{ab}$ (see \cite[Lem. 5.18 and Cor. 5.19]{Mo2}). Then the morphism $\theta^{ab}_M:M \to L_{\mathbb{A}^1}^{ab}(M)$ induces a morphism
\begin{equation}\label{MMA1}
M = H_0(M) \to H_0(L_{\mathbb{A}^1}^{ab}(M)) = M_{\mathbb{A}^1}.
\end{equation}
The morphism $\Hom_{\ModA(R)}(M_{\mathbb{A}^1},N) \to \Hom_{\Mod(R)}(M,N)$ induced by \eqref{MMA1} coincides with our adjunction.

We give an $\mathbb{A}^1$-analogue of the abelianization in ordinary group theory. This gives a left adjoint functor of $\AbA \hookrightarrow \mathcal{G}r_k^{\mathbb{A}^1}$.

\begin{defi}
Let $\mathcal{G}r_k$ be the category of Nisnevich sheaves of groups. For $G \in \mathcal{G}r_k$, we denote $G^{ab}$ for its abelianization sheaf. We call the strictly $\mathbb{A}^1$-invariant sheaf $G^{ab}_{\mathbb{A}^1} = (G^{ab})_{\mathbb{A}^1}$ the \textit{$\mathbb{A}^1$-abelianization of $G$}.
\end{defi}

\begin{lem}\label{A1-ab}
The functor $\mathcal{G}r_k \to \AbA ; G \mapsto G^{ab}_{\mathbb{A}^1}$ is left adjoint to $\AbA \hookrightarrow \mathcal{G}r_k$. Moreover, this adjunction induces a left adjoint functor of $\AbA \hookrightarrow \mathcal{G}r_k^{\mathbb{A}^1}$.
\end{lem}

\begin{proof}
Since $G \mapsto G^{ab}$ gives a left adjoint of $\Ab \hookrightarrow \mathcal{G}r_k$, the composite functor $\mathcal{G}r_k \to \Ab \to \AbA$ is left adjoint to $\AbA \hookrightarrow \mathcal{G}r_k$.
\end{proof}

\subsection{$\mathbb{A}^1$-Hurewicz theorem and Theorem \ref{main1} over an arbitrary field}

By using Lemma \ref{adj:ModA-in-Mod} and $\mathbb{A}^1$-abelianization, we can obtain a generalization of Morel's $\mathbb{A}^1$-Hurewicz theorem \cite[Thm. 5.35 and 5.37]{Mo2} for a not necessary perfect field $k$:

\begin{prop}\label{Hure}
Let $\mathcal{X}$ be a pointed $k$-space over an arbitrary field $k$.
\begin{itemize}
\item[(1)] If $\mathcal{X}$ is $\mathbb{A}^1$-connected, then there exists a natural isomorphism
\begin{equation*}
\pi_{1}^{\mathbb{A}^1}(\mathcal{X})^{ab}_{\mathbb{A}^1} \cong \mathbf{H}_1^{\mathbb{A}^1}(\mathcal{X}).
\end{equation*}
\item[(2)] If $\mathcal{X}$ is $\mathbb{A}^1$-$(m-1)$-connected for $m \geq 2$, then $\tilde{\mathbf{H}}_{i}^{\mathbb{A}^1}(\mathcal{X}) = 0$ for all $i \leq m-1$ and there exists a natural isomorphism
\begin{equation*}
\pi_{m}^{\mathbb{A}^1}(\mathcal{X})_{\mathbb{A}^1} \cong \mathbf{H}_m^{\mathbb{A}^1}(\mathcal{X}).
\end{equation*}
\end{itemize}
\end{prop}

By Lemma \ref{A1-ab}, the natural composite morphism $G \to G^{ab} \to G^{ab}_{\mathbb{A}^1}$ for $G \in \mathcal{G}r_k$ is an initial morphism in the sens of \cite[Thm. 5.35]{Mo2}. Moreover, by \cite[Cor. 5.2]{Mo2}, $\pi_{m}^{\mathbb{A}^1}(\mathcal{X})_{\mathbb{A}^1} \cong \pi_{m}^{\mathbb{A}^1}(\mathcal{X})$ for all $m \geq 2$ when $k$ is perfect. Therefore, Proposition \ref{Hure} is a generalization of Morel's $\mathbb{A}^1$-Hurewicz theorem.

\begin{proof}[Proof of Proposition \ref{Hure}]
Let $\mathcal{X}$ be an $\mathbb{A}^1$-$(m-1)$-connected pointed $k$-space for $m \geq 1$. Applying the Hurewicz theorem for simplicial sets for all stalks of $\pi_{i}^{\mathbb{A}^1}(\mathcal{X})=\pi_i(\mathrm{Ex}_{\mathbb{A}}^1(\mathcal{X}))$ and $\mathbf{H}_i(\mathrm{Ex}_{\mathbb{A}^1}(\mathcal{X})) = H_i(C(\mathrm{Ex}_{\mathbb{A}^1}(\mathcal{X})))$, we have
\begin{equation*}
\Ker(\mathbf{H}_i(\mathrm{Ex}_{\mathbb{A}^1}(\mathcal{X})) \to \mathbf{H}_i(\Spec k))=0
\end{equation*}
for all $i \leq m-1$ and $\pi_{m}^{\mathbb{A}^1}(\mathcal{X})^{ab} \cong \mathbf{H}_m(\mathrm{Ex}_{\mathbb{A}^1}(\mathcal{X}))$. Therefore, we obtain
\begin{equation}\label{Hureeq1}
\Hom_{\mathcal{G}r_k}(\pi_{m}^{\mathbb{A}^1}(\mathcal{X})^{ab},M) \cong \Hom_{\Ab}(\mathbf{H}_m(\mathrm{Ex}_{\mathbb{A}^1}(\mathcal{X}))),M)
\end{equation}
for every strictly $\mathbb{A}^1$-invariant sheaf $M$. By the same argument of the proof of Lemma \ref{key} (not taking the $\mathbb{A}^1$-localization of $C(\mathrm{Ex}_{\mathbb{A}^1}(\mathcal{X}))$), the right-hand side of \eqref{Hureeq1} coincides with $[\mathrm{Ex}_{\mathbb{A}^1}(\mathcal{X}),K(M,m)]_s$. Since the natural morphism $\mathcal{X} \to \mathrm{Ex}_{\mathbb{A}^1}(\mathcal{X})$ is an $\mathbb{A}^1$-weak equivalence and $K(M,m)$ is $\mathbb{A}^1$-local, we also have $[\mathrm{Ex}_{\mathbb{A}^1}(\mathcal{X}),K(M,m)]_s \cong [\mathcal{X},K(M,m)]_s$. On the other hand, by Lemma \ref{adj:ModA-in-Mod}, the left-hand side of \eqref{Hureeq1} coincides with $\Hom_{\AbA}(\pi_{m}^{\mathbb{A}^1}(\mathcal{X})^{ab}_{\mathbb{A}^1},M)$. Therefore, we have
\begin{equation}\label{Hureeq2}
\Hom_{\AbA}(\pi_{m}^{\mathbb{A}^1}(\mathcal{X})^{ab}_{\mathbb{A}^1},M) \cong [\mathcal{X},K(M,m)]_s.
\end{equation}
Since $\tilde{\mathbf{H}}_{0}^{\mathbb{A}^1}(\mathcal{X}) = 0$, the proof of the Lemma \ref{key} shows
\begin{equation*}
[\mathcal{X},K(M,1)]_s \cong \Hom_{\AbA}(\mathbf{H}_1^{\mathbb{A}^1}(\mathcal{X}),M)
\end{equation*}
and \eqref{Hureeq2} induces
\begin{equation*}
\Hom_{\AbA}(\pi_{1}^{\mathbb{A}^1}(\mathcal{X})^{ab}_{\mathbb{A}^1},M) \cong \Hom_{\AbA}(\mathbf{H}_1^{\mathbb{A}^1}(\mathcal{X}),M).
\end{equation*}
Therefore, Yoneda's lemma in $\AbA$ proves (1). Assume $m\geq2$. Then $\pi_{m}^{\mathbb{A}^1}(\mathcal{X})^{ab}_{\mathbb{A}^1}=\pi_{m}^{\mathbb{A}^1}(\mathcal{X})_{\mathbb{A}^1}$. Therefore, for proving (2), we only need to show that
\begin{equation}\label{Hureeq3}
\pi_{m}^{\mathbb{A}^1}(\mathcal{X})^{ab}_{\mathbb{A}^1} \cong \mathbf{H}_m^{\mathbb{A}^1}(\mathcal{X})
\end{equation}
by using induction on $m$. By the inductive hypothesis, we may assume $\tilde{\mathbf{H}}_{i}^{\mathbb{A}^1}(\mathcal{X}) = 0$ for every $i \leq m-1$ because $\tilde{\mathbf{H}}_{m-1}^{\mathbb{A}^1}(\mathcal{X}) \cong \pi_{m-1}^{\mathbb{A}^1}(\mathcal{X})_{\mathbb{A}^1}=0$. Then \eqref{Hureeq2} also shows
\begin{equation*}
\Hom_{\AbA}(\pi_{m}^{\mathbb{A}^1}(\mathcal{X})^{ab}_{\mathbb{A}^1},M) \cong \Hom_{\AbA}(\mathbf{H}_m^{\mathbb{A}^1}(\mathcal{X}),M).
\end{equation*}
Therefore, Yoneda's lemma in $\AbA$ also proves \eqref{Hureeq3}.
\end{proof}

The following is a weaker version of Theorem \ref{main1}.

\begin{thm}\label{Finalcor}
Let $X$ be a pointed smooth $k$-scheme of dimension $n \geq 1$. If $X$ is $\mathbb{A}^1$-$m$-connected for $m \geq n$, then $\pi_{m+1}^{\mathbb{A}^1}(X)_{\mathbb{A}^1} = 0$.
\end{thm}

\begin{proof}
We can perform the same argument of the proof of Theorem \ref{main1} replacing $\pi_{m+1}^{\mathbb{A}^1}(X)$ to $\pi_{m+1}^{\mathbb{A}^1}(X)_{\mathbb{A}^1}$. Indeed, Proposition \ref{Hure} and Lemma \ref{key} show
\begin{equation*}
\Hom_{\AbA}(\pi_{m+1}^{\mathbb{A}^1}(X)_{\mathbb{A}^1},M) \cong H^{m+1}_{Nis}(X,M)
\end{equation*}
for every $M \in \AbA$. By \cite[Thm. 1.32]{Nis}, the right-hand side vanishes. Since $\pi_{m+1}^{\mathbb{A}^1}(X)_{\mathbb{A}^1} \in \AbA$, Yoneda's lemma in $\AbA$ shows $\pi_{m+1}^{\mathbb{A}^1}(X)_{\mathbb{A}^1} = 0$.
\end{proof}


\end{document}